\documentclass{amsart}
\usepackage{amsmath, amssymb, amscd, mathtools}

\newtheorem{theorem}{Theorem}[section]
\newtheorem{proposition}[theorem]{Proposition}
\newtheorem{lemma}[theorem]{Lemma}
\newtheorem{corollary}[theorem]{Corollary}
\newtheorem{conjecture}[theorem]{Conjecture}
\newtheorem{question}[theorem]{Question}

\theoremstyle{definition}

\newtheorem{example}[theorem]{Example}
\newtheorem{remark}[theorem]{Remark}

\numberwithin{equation}{section}

\newcommand{\N}{\mathbb{N}}                        
\newcommand{\K}{\mathbb{K}}                        
\newcommand{\R}{\mathbb{R}}                        
\newcommand{\C}{\mathbb{C}}                        
\newcommand{\vp}{\varphi}                          
\newcommand{\MM}{\mathcal{M}}
\newcommand{\VV}{\mathcal{V}}                      
\newcommand{\OO}{\mathcal{O}}                      
\newcommand{\XX}{\mathfrak{X}}                     
\newcommand{\fbd}{\mathrm{fbd}}                    
\newcommand{\OXxi}{{\mathcal{O}}_{X,\xi}}          
\newcommand{\supp}{\mathrm{supp}}                  
\newcommand{\inform}{\mathrm{in}}                  
\newcommand{\inexp}{\mathrm{exp}}                  
\newcommand{\Pt}{\Phi^\theta}
\newcommand{\NN}{\mathfrak{N}}                     
\newcommand{\mm}{\mathfrak{m}}                     

\begin{document}

\title{Finite determinacy and stability of flatness of analytic mappings}

\author{Janusz Adamus}
\address{Department of Mathematics, The University of Western Ontario, London, Ontario, Canada N6A 5B7 -- and --
         Institute of Mathematics, Polish Academy of Sciences, ul. {\'S}niadeckich 8, 00-956 Warsaw, Poland}
\email{jadamus@uwo.ca}
\author{Hadi Seyedinejad}
\address{Department of Mathematical Sciences, University of Kashan, Kashan, Iran 8731753153}
\email{sseyedin@alumni.uwo.ca}
\thanks{J. Adamus's research was partially supported by the Natural Sciences and Engineering Research Council of Canada}


\begin{abstract}
It is proved that flatness of an analytic mapping germ from a complete intersection is determined by its sufficiently high jet. As a consequence, one obtains finite determinacy of complete intersections. It is also shown that flatness and openness are stable under deformations.
\end{abstract}
\maketitle

\section{Introduction}
\label{sec:intro}

When dealing with singularities of analytic sets or mappings, particularly in explicit calculations, one is often tempted to forget the original infinite transcendental data and to work instead with its (sufficiently long) Taylor truncation. This approach is satisfactory in many circumstances. For example, the Milnor number of an isolated hypersurface singularity can be correctly calculated this way.
In general, however, local analytic invariants of a given singularity may differ from those of its Taylor approximations of arbitrary length (cf. Example~\ref{ex:no-equality} and Remark~\ref{rem:same-as-HS-function}).

In the present paper we show that, roughly speaking, those algebro-geometric properties of an analytic mapping-germ $\vp=(\vp_1,\dots,\vp_n)$ that depend on the variation of its fibres are determined already by Taylor polynomials of $\vp_1,\dots,\vp_n$ of sufficiently large degree.

\subsection{Main results}

Let $\K=\R$ or $\C$. Let $x=(x_1,\dots,x_m)$ and let $\mm_x$ denote the maximal ideal in the ring of convergent power series $\K\{x\}$. For a natural number $k\in\N$ and a power series $f\in\K\{x\}$, the \emph{$k$-jet} of $f$, denoted $j^k f$, is the image of $f$ under the canonical epimorphism $\K\{x\}\to\K\{x\}/\mm_x^{k+1}$. For an $s$-tuple $\vp=(\vp_1,\dots,\vp_s)\in\K\{x\}^s$, we set $j^k\vp=(j^k\vp_1,\dots,j^k\vp_s)$.

An analytic mapping $\vp:X\to Y$ is called \emph{flat} at a point $\xi\in X$ when the pullback homomorphism $\vp^*_\xi:\OO_{Y,\vp(\xi)}\to\OO_{X,\xi}$ makes $\OO_{X,\xi}$ into a flat $\OO_{Y,\vp(\xi)}$-module.
One of the major problems considered in this paper is whether flatness of a $\K$-analytic mapping germ is finitely determined (i.e., determined by its $k$-jet for $k$ large enough). We prove that this is indeed the case for mappings from complete intersections. More precisely, we have the following.

\begin{theorem}
\label{thm:flat-fin-det-comp-int}
Let $X$ be a $\K$-analytic subspace of $\K^m$. Suppose that $0\in X$ and the local ring $\OO_{X,0}$ is a complete intersection.
If $\vp=(\vp_1,\dots,\vp_n):X\to\K^n$, $\vp(0)=0$, is a $\K$-analytic mapping, then there exists $\mu_0\geq1$ such that the following conditions are equivalent:
\begin{itemize}
\item[(i)] $\vp$ is flat at zero.
\item[(ii)] For every $\mu\geq\mu_0$, $j^\mu\vp=(j^\mu\vp_1,\dots,j^\mu\vp_n):X\to\K^n$ is flat at zero.
\item[(iii)] There exists $\mu\geq\mu_0$ such that $j^\mu\vp=(j^\mu\vp_1,\dots,j^\mu\vp_n):X\to\K^n$ is flat at zero.
\end{itemize}
\end{theorem}

The above theorem implies, in particular, \emph{finite determinacy of complete intersections in $\K\{x\}$}: An $s$-tuple $(f_1,\dots,f_s)$ forms a regular sequence in $\K\{x\}$ if and only if this is so for every $(g_1,\dots,g_s)\in\K\{x\}^s$ with $j^{\mu_0}g_i=j^{\mu_0}f_i$, $i=1,\dots,s$ (Corollary~\ref{cor:fin-det-comp-int}).
As a consequence, we obtain that $\vp:X\to\K^n$ is a flat mapping if and only if every $\psi:X\to\K^n$ satisfying $j^{\mu_0}\psi=j^{\mu_0}\vp$ is flat as well (Theorem~\ref{thm:2flat-fin-det-comp-int}).

Independently, one can prove an analogue of Theorem~\ref{thm:flat-fin-det-comp-int} for finite mappings. This is our Theorem~\ref{thm:flat-fin-det-finite}, which requires no assumptions on the source $X$. Interestingly, Theorem~\ref{thm:flat-fin-det-finite} cannot be generalized to finitely generated modules. That is, in general, flatness of a finitely generated $\K\{x\}$-module is not finitely determined (Example~\ref{ex:not-for-modules}).
\medskip

Another problem considered here is that of stability of flatness under deformations. Recall that given a morphish $\vp:X\to Y$ of $\K$-analytic spaces and a pointed space $(T,0)$ (i.e., an analytic space with a distinguished point $0$), a \emph{deformation of $\vp$ over $Y$} is a Cartesian diagram of the form
\[
\begin{CD}
X    @>>>    \XX \\
@V{\vp}VV    @VV{\Phi}V \\
Y     @>>>   Y\times T \\
@VVV         @VV{\pi}V \\
\{0\} @>>>   T,
\end{CD}
\]
such that $\XX$ is flat over $T$ (see, e.g., \cite{GLS}).
For $\theta\in T$, we denote by $\Pt$ the specialization of $\Phi$ over $\theta$, that is, the pullback of $\Phi$ by $Y\times\{\theta\}\hookrightarrow Y\times T$.

We prove that flatness of mappings into locally irreducible targets is stable under deformations:

\begin{theorem}
\label{thm:stab-flat}
Let $\vp:X\to Y$ be a flat morphism of $\K$-analytic spaces, with $Y$ locally irreducible. If $\Phi$ is a deformation of $\vp$ parametrized by a locally irreducible $T$, then $\Pt$ is a flat mapping for every $\theta\in T$ near zero.
\end{theorem}

When $\K=\C$, we also have an analogous result about stability of openness (Proposition~\ref{prop:stab-defect} and Corollary~\ref{cor:stab-open}).
\medskip

\subsection{Plan of the paper}

Our main tool here is Hironaka's diagram of initial exponents. We recall this notion and its relevance to flatness in the next section. Section~\ref{sec:approx} is devoted to approximation of the diagram of a given ideal in $\K\{x\}$ by the diagrams of its Taylor approximations. Theorems~\ref{thm:flat-fin-det-comp-int} and~\ref{thm:stab-flat} are proved in Sections~\ref{sec:determinacy} and~\ref{sec:stability} respectively. The last section is devoted to the complex case. In the complex setting, we prove geometric analogues of our main results.
\bigskip

\section{Diagram of initial exponents and flatness}
\label{sec:toolbox}

Let $\K=\R$ or $\C$. Let $A=\K\{y\}/J$ be a local analytic $\K$-algebra, where $y=(y_1,\dots,y_n)$ and $J\subset\K\{y\}$ is a proper ideal.
Let $x=(x_1,\dots,x_m)$ and define $A\{x\}\coloneqq\K\{y,x\}/J\!\cdot\!\K\{y,x\}$.
We will write $x^\beta$ for $x_1^{\beta_1}\dots x_m^{\beta_m}$, where $\beta=(\beta_1,\dots,\beta_m)\in\N^m$.
By the canonical embedding $A\{x\}\hookrightarrow A[[x]]$, one can regard the elements of $A\{x\}$ as power series with coefficients in $A$.

Let $\mm_A$ denote the maximal ideal of $A$.
For a power series $F=\sum_{\beta\in\N^m}f_\beta x^\beta\in A\{x\}$, define its \emph{evaluation at $0$} as
$F(0)=\sum_{\beta\in\N^m}f_\beta(0)x^\beta\in A/\mm_A\{x\}=\K\{x\}$, and for an ideal $I$ in $A\{x\}$
define $I(0)=\{F(0): F\in I\}$, the \emph{evaluated} ideal.

For $\beta=(\beta_1,\dots,\beta_m)\in\N^m$, we write $|\beta|=\beta_1+\dots+\beta_m$ and call it the \emph{length} of $\beta$.
We define a total ordering of $\N^m$ by lexicographic ordering of the $(m+1)$-tuples $(|\beta|,\beta_1,\dots,\beta_m)$, where $\beta=(\beta_1,\dots,\beta_m)$.
The \emph{support} of $F=\sum_{\beta\in\N^m}f_\beta x^\beta$ is defined as $\supp(F)=\{\beta\in\N^m:f_\beta\neq0\}$,
and $\inexp(F)=\min\{\beta\in\supp(F)\}$ denotes the \emph{initial exponent} of $F$.
Similarly, $\supp(F(0))=\{\beta\in\N^m:f_\beta(0)\neq0\}$ and $\inexp(F(0))=\min\{\beta\in\supp(F(0))\}$, for the evaluated
series. Of course, $\supp(F(0))\subset\supp(F)$.

If $\beta^*=\inexp(F)$ is the initial exponent of $F=\sum f_\beta x^\beta$, then the `monomial' $f_{\beta^*}x^{\beta^*}\in A\{x\}$ is called the \emph{initial term} of $F$ and denoted $\inform(F)$.

Given an ideal $I$ in $A\{x\}$, we denote by $\NN(I)$ the \emph{diagram of initial exponents} of $I$, that is,
\[
\NN(I)=\{\inexp(F):F\in I\setminus\{0\}\}\,.
\]
Similarly, for the evaluated ideal $I(0)$, we set
\[
\NN(I(0))=\{\inexp(F(0)):F\in I, F(0)\neq0\}\,.
\]
Note that every diagram $\NN(I)$ satisfies the equality $\NN(I)+\N^m=\NN(I)$. (Indeed, for $\beta\in\NN(I)$ and $\gamma\in\N^m$, one can choose $F\in I$ such that $\inexp(F)=\beta$; then $x^\gamma F\in I$ and hence $\beta+\gamma=\inexp(x^\gamma F)$ is in $\NN(I)$.)

\begin{remark}
\label{rem:vertices}
It is not difficult to show that, for every ideal $I$, there exists a unique smallest (finite) set $V(I)\subset\NN(I)$ such that $V(I)+\N^m=\NN(I)$ (see, e.g., \cite[Lem.\,3.8]{BM}). The elements of $V(I)$ are called the \emph{vertices} of the diagram $\NN(I)$.
\end{remark}

We now recall Hironaka's combinatorial criterion that expresses flatness in terms of the diagram of initial exponents:
For an ideal $I$ in $A\{x\}$, set $\Delta=\N^m\setminus\NN(I(0))$, and define $A\{x\}^\Delta=\{F\in A\{x\}:\supp(F)\subset\Delta\}$.
Consider the canonical projection $A\{x\}\to A\{x\}/I$ and its restriction to $A\{x\}^\Delta$, called $\kappa$. The following two results will be used throughout the paper.

\begin{proposition}[{\cite[\S\,6, Prop.\,9]{H}}]
\label{prop:Hir1}
The mapping $\kappa:A\{x\}^\Delta\to A\{x\}/I$ is surjective.
\end{proposition}

\begin{proposition}[{\cite[\S\,6, Prop.\,10]{H}}]
\label{prop:Hir2}
The ring $A\{x\}/I$ is flat as an $A$-module if and only if $\kappa$ is bijective.
\end{proposition}

\section{Approximation of diagrams}
\label{sec:approx}

Let $\K=\R$ or $\C$. Let $x=(x_1,\dots,x_m)$ and let $\mm_x$ denote the maximal ideal of $\K\{x\}$. Recall that, for a natural number $k\in\N$ and a power series $f\in\K\{x\}$, the \emph{$k$-jet} of $f$ (denoted $j^k f$) is the image of $f$ under the canonical epimorphism $\K\{x\}\to\K\{x\}/\mm_x^{k+1}$.

In the present section we study the relations between the diagram of initial exponents of a given ideal in $\K\{x\}$ and those of its Taylor approximations. Throughout this section, we will use the following notation: Let $f_1,\dots,f_s$ be a finite collection of power series in $\K\{x\}$ and let
\[
I=(f_1,\dots,f_s)\cdot\K\{x\}\,.
\]
For a natural number $\mu$, let $I_\mu$ denote the ideal generated by the $\mu$-jets $j^\mu{f_i}$, $i=1,\dots,s$, that is,
\[
I_\mu=(j^\mu{f_1},\dots,j^\mu{f_s})\cdot\K\{x\}\,.
\]

The following simple observation will be used often in our considerations.

\begin{remark}
\label{rem:jet-exp}
Given a power series $F\in\K\{x\}$, suppose that $\mu\geq|\inexp(F)|$.
Then 
\[
\inexp(F)=\inexp(G)
\]
for every $G\in\K\{x\}$ with $j^\mu G=j^\mu F$.
\end{remark}

We now show the connection between the diagram of initial exponents of $I$ and those of its approximations $I_\mu$.

\begin{lemma}
\label{lem:diagram-up-to-l}
Let $I$ and $\{I_\mu\}_{\mu\in\N}$ be as above. Let $l_0$ be the maximum of lengths of vertices of the diagram $\NN(I)$. Then:
\begin{itemize}
\item[(i)] $\NN(I_\mu)\supset\NN(I)$ for all $\mu\geq l_0$.
\item[(ii)] For every $l\geq l_0$, 
\[
\NN(I_\mu)\cap\{\beta\in\N^m:|\beta|\leq l\}=\NN(I)\cap\{\beta\in\N^m:|\beta|\leq l\}\,,
\]
for all $\mu\geq l$.
\end{itemize}
\end{lemma}

\begin{proof}
Fix $\mu\geq l_0$. By Remark~\ref{rem:vertices}, for the proof of (i) it suffices to show that the vertices of $\NN(I)$ are contained in $\NN(I_\mu)$.
Let $F_1,\dots,F_q\in I$ be any set of representatives of the vertices of $\NN(I)$. We can write $F_k=\sum_{i=1}^sg_{ki}f_i$, $k=1,\dots,q$, for some $g_{ki}\in\K\{x\}$. Then
\[
j^\mu{F_k}=j^\mu(\sum_{i=1}^sg_{ki}f_i)=j^\mu(\sum_{i=1}^sg_{ki}\!\cdot\!j^\mu{f_i})\,,
\]
since the power series of a product up to order $\mu$ depends only on the power series up to order $\mu$ of its factors.
Hence, by Remark~\ref{rem:jet-exp}, we have equality of the initial exponents $\inexp(F_k)=\inexp(\sum_{i=1}^sg_{ki}\!\cdot\!j^\mu{f_i})$. It follows that $\inexp(F_k)\in\NN(I_\mu)$ for all $k$, which proves (i).

For the proof of part (ii), fix $l\geq l_0$. It now suffices to show that
\[
\NN(I_\mu)\cap\{\beta\in\N^m:|\beta|\leq l\}\ \subset\ \NN(I)\cap\{\beta\in\N^m:|\beta|\leq l\}\,,
\]
for every $\mu\geq l$. Pick $\beta^*\in\N^m\setminus\NN(I)$ with $|\beta^*|\leq l$. Suppose that $\beta^*\in\NN(I_\mu)$ for some $\mu\geq l$. Then one can choose $G\in I_\mu$ with $\inexp(G)=\beta^*$. Write $G=\sum_{i=1}^sg_i\cdot j^\mu{f_i}$ for some $g_i\in\K\{x\}$.
We have $j^\mu G=j^\mu(\sum_{i=1}^sg_i\cdot j^\mu{f_i})=j^\mu(\sum_{i=1}^sg_if_i)$, and since $|\inexp(G)|=|\beta^*|\leq\mu$, it follows that $\inexp(G)=\inexp(\sum_{i=1}^sg_if_i)$, by Remark~\ref{rem:jet-exp} again. Therefore $\beta^*\in\NN(I)$; a contradiction.
\end{proof}

\begin{corollary}
\label{cor:1diagram}
Let $I$ and $\{I_\mu\}_{\mu\in\N}$ be as above. Then
\[
\NN(I)=\bigcap_{\mu\geq l_0}\NN(I_\mu)\,,
\]
where $l_0$ is the maximum of lengths of vertices of the diagram $\NN(I)$.
\end{corollary}

\begin{proof}
By Lemma~\ref{lem:diagram-up-to-l}(i), we have $\NN(I)\subset\bigcap_{\mu\geq l_0}\NN(I_\mu)$.
On the other hand, by part (ii) of the lemma, if $\beta^*\in\N^m\setminus\NN(I)$, then $\beta^*\notin\NN(I_\mu)$ for all $\mu\geq\max\{|\beta^*|,l_0\}$.
\end{proof}

\begin{corollary}
\label{cor:2diagram}
Let $I$ and $\{I_\mu\}_{\mu\in\N}$ be as above. Let $l_0$ be the maximum of lengths of vertices of the diagram $\NN(I)$. If $I$ contains a $k$'th power of the maximal ideal $\mm_x$, then $k\geq l_0$ and $\NN(I_\mu)=\NN(I)$ for all $\mu\geq k$.
\end{corollary}

\begin{proof}
If $I\supset\mm_x^k$ for some $k$, then $\NN(I)\supset\NN(\mm_x^k)$. Hence $l_0\leq k$ and
\[
(\N^m\setminus\NN(I))\cap\{\beta\in\N^m:|\beta|\leq l\}\ =\ \N^m\setminus\NN(I)\,,
\]
for all $l\geq k$. The statement thus follows from Lemma~\ref{lem:diagram-up-to-l}(ii).
\end{proof}

It is important to observe that, in general, there need not be equality between the diagrams of $I$ and $I_\mu$, for $\mu$ arbitrarily large. This is shown in the following example.

\begin{example}
\label{ex:no-equality}
Let $I$ be an ideal in $\K\{x,y\}$ generated by $f_1$ and $f_2$ of the form
\begin{align*}
f_1 &= x^3y+xy^4+xy^5+xy^6+\dots,\\
f_2 &= x^2y^3+y^6+y^7+y^8+\dots.
\end{align*}
Then, for every $\mu\geq5$, we have $y^2\cdot j^\mu f_1-x\cdot j^\mu f_2=xy^{\mu+1}$, hence $(1,\mu+1)\in\NN(I_\mu)$. However, $(1,k)\notin\NN(I)$ for any $k\geq 1$.

We prove the latter by contradiction. Suppose there exists $F\in I$ with $\inexp(F)=(1,k_0)$ for some $k_0\in\N$. Choose $h_1,h_2\in\K\{x,y\}$ so that $F=h_1f_1+h_2f_2$. Let $ax^{\alpha_1}y^{\alpha_2}$ and $bx^{\beta_1}y^{\beta_2}$ be the initial terms of $h_1$ and $h_2$ respectively. Clearly, $\inform(h_1)\cdot\inform(f_1)+\inform(h_2)\cdot\inform(f_2)=0$, for otherwise the $x$-component of $\inexp(h_1f_1+h_2f_2)$ would not be $1$. Therefore, $ax^{\alpha_1+3}y^{\alpha_2+1}+bx^{\beta_1+2}y^{\beta_2+3}=0$. It follows that $\alpha_1+1=\beta_1$, $\alpha_2=\beta_2+2$, and $a+b=0$. Consequently,
\begin{equation}
\label{eq:zero}
\inform(h_1)\cdot f_1+\inform(h_2)\cdot f_2=0.
\end{equation}
Now, set $h_i^{(1)}\coloneqq h_i-\inform(h_i)$, $i=1,2$. By \eqref{eq:zero}, we get $h_1^{(1)}f_1+h_2^{(1)}f_2=F$. Hence, by repeating the above argument, $\inform(h_1^{(1)})\cdot f_1+\inform(h_2^{(1)})\cdot f_2=0$. We can thus set $h_i^{(2)}\coloneqq h_i^{(1)}-\inform(h_i^{(1)})$, $i=1,2$, and again obtain $h_1^{(2)}f_1+h_2^{(2)}f_2=F$.
By induction, if $h_i^{(j)}=h_i^{(j-1)}-\inform(h_i^{(j-1)})$, $i=1,2$, then
\begin{equation}
\label{eq:j}
h_1^{(j)}f_1+h_2^{(j)}f_2=F,\mathrm{\ for\ all\ } j.
\end{equation}
Note that, for every $j\geq1$, the initial exponent of $h_i^{(j+1)}$ is strictly greater than that of $h_i^{(j)}$, by construction.
Therefore, by the Krull Intersection Theorem, the sequences $(h_1^{(j)})_{j\geq1}$ and $(h_2^{(j)})_{j\geq1}$ converge to zero in the Krull topology of $\K\{x,y\}$. It follows from \eqref{eq:j} that $0\!\cdot\!f_1+0\!\cdot\!f_2=F$, hence $F=0$, which contradicts the choice of $F$.\qed
\end{example}

\begin{remark}
\label{rem:same-as-HS-function}
Identifying those ideals $I$ (or, more precisely, those systems of generators) in $\K\{x\}$ for which the sequence of diagrams $(\NN(I_\mu))_\mu$ stabilizes seems important from the point of view of singularity theory. In fact, the singularities defined by such ideals are precisely those whose Hilbert-Samuel function is finitely determined. Indeed, for an ideal $J$ in $\K\{x\}$, let $H_J(k)=\dim_\K\K\{x\}/(J+\mm_x^{k+1})$ denote the Hilbert-Samuel function of $\K\{x\}/J$. It follows from Proposition~\ref{prop:Hir1} that
\[
H_J(k)\,=\,\#\,(\N^m\setminus\NN(J))\cap\{\beta\in\N^m:|\beta|\leq k\}.
\]
Therefore, by Lemma~\ref{lem:diagram-up-to-l}(i), we have $\NN(I)=\NN(I_\mu)$ if and only if $H_I(k)=H_{I_\mu}(k)$ for all $k\in\N$.
\end{remark}

We conjecture that the diagrams of $I_\mu$ stabilize in case when $I$ is generated by a regular sequence.
Recall that an $s$-tuple $(f_1,\dots,f_s)$ forms a $\K\{x\}$-regular sequence when $f_1,\dots,f_s$ generate a proper ideal, $f_1\neq0$, and $f_{i+1}$ is not a zerodivisor in $\K\{x\}/(f_1,\dots,f_i)$ for $i=1,\dots,s-1$.

\begin{conjecture}
\label{conj:diagram-fin-det}
Let $I=(f_1,\dots,f_s)$ be an ideal in $\K\{x\}$, and for $\mu\in\N$, let $I_\mu$ denote the ideal generated in $\K\{x\}$ by the $\mu$-jets $j^\mu f_1,\dots,j^\mu f_s$. Suppose that $f_1,\dots,f_s$ form a regular sequence in $\K\{x\}$. Then, there exists $\mu_0\geq1$ such that $\NN(I)=\NN(I_\mu)$, for all $\mu\geq\mu_0$ (after a generic change of coordinates $x$, if needed).
\end{conjecture}

The above conjecture is partly justified by the following observation:
Let $I$ and $\{I_\mu\}_{\mu\in\N}$ be as above.
If $f_1,\dots,f_s$ form a regular sequence in $\K\{x\}$, then there exists $\mu_0\geq1$ such that
\[
\dim(\K\{x\}/I)=\dim(\K\{x\}/I_\mu)\,,
\]
for all $\mu\geq\mu_0$.

Indeed, as an ideal generated by $s$ elements in an $m$-dimensional ring, every $I_\mu$ satisfies inequality
\[
m-s\leq\dim(\K\{x\}/I_\mu).
\]
On the other hand, by \cite[Ch.\,2, Prop.\,5.3]{Tou}, there exists $\mu_0\geq1$ such that
\[
\dim(\K\{x\}/I_\mu)\,\leq\,\dim(\K\{x\}/I)\,,
\]
for all $\mu\geq\mu_0$. Finally, if $f_1,\dots,f_s$ form a regular sequence, then $\dim(\K\{x\}/I)=m-s$.

\section{Finite determinacy of flatness}
\label{sec:determinacy}

We now turn to the problem of finite determinacy of flatness of analytic mappings.
Throughout this section $X$ will denote a $\K$-analytic space, and $\vp:X\to\K^n$ will be a $\K$-analytic mapping. Since our considerations are local, we can assume without loss of generality that $X$ is a subspace of $\K^m$ (for some $m\geq1$), $0\in X$, and $\vp(0)=0$.
Let us begin with finite mappings.

\subsection{Finite mappings}

\begin{proposition}
\label{prop:flat-fin-det-finite}
Let $X$ be a $\K$-analytic subspace of $\K^m$, with $0\in X$ and $\OO_{X,0}=\K\{x\}/I$. Let $\vp=(\vp_1,\dots,\vp_n):X\to\K^n$, $\vp(0)=0$, be a finite $\K$-analytic mapping. Then, there exists $\mu_0\geq1$ such that the following conditions are equivalent:
\begin{itemize}
\item[(i)] $\vp$ is flat at zero.
\item[(ii)] For every $\mu\geq\mu_0$, $(j^\mu\vp_1,\dots,j^\mu\vp_n):X\to\K^n$ is flat at zero.
\item[(iii)] There exists $\mu\geq\mu_0$ such that $(j^\mu\vp_1,\dots,j^\mu\vp_n):X\to\K^n$ is flat at zero.
\end{itemize}
If, moreover, $I=(h_1,\dots,h_s)$ and, for $\mu\in\N$, $X_\mu$ denotes a local model defined at $0\in\K^m$ by $\OO_{X_\mu,0}=\K\{x\}/(j^\mu h_1,\dots,j^\mu h_s)$, then the above conditions are equivalent to the following:
\begin{itemize}
\item[(ii')] For all $\mu\geq\mu_0$, $(j^\mu\vp_1,\dots,j^\mu\vp_n):X_\mu\to\K^n$ is flat at zero.
\item[(iii')] There exists $\mu\geq\mu_0$ such that $(j^\mu\vp_1,\dots,j^\mu\vp_n):X_\mu\to\K^n$ is flat at zero.
\end{itemize}
\end{proposition}

\begin{proof}
Identifying $X$ with the graph of $\vp$, we can write $\OO_{X,0}=\K\{y,x\}/J$, where $y=(y_1,\dots,y_n)$, $x=(x_1,\dots,x_m)$, and
\[
J=(h_1(x),\dots,h_s(x),y_1-\vp_1(x),\dots,y_n-\vp_n(x))\cdot\K\{y,x\}\,.
\]
For $\mu\geq1$, set
\[
\tilde{J}_\mu\coloneqq(h_1,\dots,h_s,y_1-j^\mu\vp_1,\dots,y_n-j^\mu\vp_n)\cdot\K\{y,x\}\,,
\]
and
\[
J_\mu\coloneqq(j^\mu{h_1},\dots,j^\mu{h_s},y_1-j^\mu\vp_1,\dots,y_n-j^\mu\vp_n)\cdot\K\{y,x\}\,.
\]
The finiteness of $\vp$ implies that $J(0)\supset\mm_x^k$ for some integer $k$ (where the evaluation is at $y=0$). Let $\mu_0$ denote the least such $k$. We shall prove that the theorem holds with this choice of $\mu_0$.

Let $l_0$ denote the maximum of lengths of the vertices of $\NN(J(0))$. By the proof of Lemma~\ref{lem:diagram-up-to-l}(i), we have
\[
\NN(J(0))\ \subset\ \NN(\tilde{J}_\mu(0))\ \subset\ \NN(J_\mu(0))\,,
\]
for all $\mu\geq l_0$. Hence, by Corollary~\ref{cor:2diagram},
\[
\NN(J(0))\ =\ \NN(\tilde{J}_\mu(0))\ =\ \NN(J_\mu(0))\,,
\]
for all $\mu\geq\mu_0$. Note also that $|\beta|\leq\mu_0$ for all $\beta\in\N^m\setminus\NN(J(0))$.

For the proof of (i)\,$\Rightarrow$\,(ii), suppose there exists $\mu\geq\mu_0$ such that $(j^\mu\vp_1,\dots,j^\mu\vp_n):X\to\K^n$ is not flat at zero. Then, by Proposition~\ref{prop:Hir2}, one can choose a nonzero $G\in\tilde{J}_\mu$ with $\inexp(G)\in\N^m\setminus\NN(\tilde{J}_\mu(0))$. Write $G=\sum_{i=1}^sg_i\cdot h_i+\sum_{j=1}^nq_j\cdot(y_j-j^\mu\vp_j)$, for some $g_i,q_j\in\K\{y,x\}$.
Define $F:=\sum_{i=1}^sg_i\cdot h_i+\sum_{j=1}^nq_j\cdot(y_j-\vp_j)$. Then $F\in J$, and $\inexp(F)=\inexp(G)$ (by Remark~\ref{rem:jet-exp}), hence $\inexp(F)\in\N^m\setminus\NN(J(0))$. This proves that $\vp$ is not flat at zero, by Proposition~\ref{prop:Hir2} again.

The implication (ii)\,$\Rightarrow$\,(iii) is trivial.
To prove that (iii)\,$\Rightarrow$\,(i), suppose that $\vp$ is not flat at zero, and choose $F=\sum_{i=1}^sg_i\cdot h_i+ \sum_{j=1}^nq_j\cdot(y_j-\vp_j)$ with $\inexp(F)\in\N^m\setminus\NN(J(0))$. Given an arbitrary $\mu\geq\mu_0$, set $G:=\sum_{i=1}^sg_i\cdot h_i+\sum_{j=1}^nq_j\cdot(y_j-j^\mu\vp_j)$. Then $G\in\tilde{J}_\mu$ and $\inexp(G)=\inexp(F)\in\N^m\setminus\NN(\tilde{J}_\mu(0))$, which proves that $(j^\mu\vp_1,\dots,j^\mu\vp_n):X\to\K^n$ is not flat at zero.

The proof of implications (i)\,$\Rightarrow$\,(ii') and (iii')\,$\Rightarrow$\,(i) is analogous, with ideal $\tilde{J}_\mu$ replaced by $J_\mu$.
\end{proof}

\begin{remark}
\label{rem:Milnor-will-suffice}
Note that in Proposition~\ref{prop:flat-fin-det-finite} one can take $\mu_0$ to be the Milnor number of the fibre $\vp^{-1}(0)$; i.e.,
\[
\mu_0=\dim_\K(\K\{x\}/J(0))\,.
\]
Indeed, since $\dim_\K(\K\{x\}/J(0))$ is precisely the cardinality of the complement of $\NN(J(0))$, one readily sees that with this choice of $\mu_0$ we have $J(0)\supset\mm_x^{\mu_0}$.
\end{remark}

Proposition~\ref{prop:flat-fin-det-finite} implies that flatness of finite analytic mappings is finitely determined, in the following sense.

\begin{theorem}
\label{thm:flat-fin-det-finite}
Let $X$ be a $\K$-analytic subspace of $\K^m$, with $0\in X$ and $\OO_{X,0}=\K\{x\}/I$. Let $\vp=(\vp_1,\dots,\vp_n):X\to\K^n$, $\vp(0)=0$, be a finite $\K$-analytic mapping. Then, there exists $\mu_0\geq1$ such that the following conditions are equivalent:
\begin{itemize}
\item[(i)] $\vp$ is flat at zero.
\item[(ii)] For every $\mu\geq\mu_0$ and for every analytic mapping $\psi=(\psi_1,\dots,\psi_n):X\to\K^n$ with $j^\mu\psi_j=j^\mu\vp_j$, $j=1,\dots,n$, $\psi$ is flat at zero.
\item[(iii)] There exist $\mu\geq\mu_0$ and an analytic mapping $\psi=(\psi_1,\dots,\psi_n):X\to\K^n$ such that $\psi$ is flat at zero and $j^\mu\psi_j=j^\mu\vp_j$, $j=1,\dots,n$.
\end{itemize}
If, moreover, $I=(h_1,\dots,h_s)$ and, for $\mu\in\N$, $X_\mu$ denotes a local model defined at $0\in\K^m$ by $\OO_{X_\mu,0}=\K\{x\}/(j^\mu h_1,\dots,j^\mu h_s)$, then the above conditions are equivalent to the following:
\begin{itemize}
\item[(ii')] For every $\mu\geq\mu_0$ and for every analytic mapping $\psi=(\psi_1,\dots,\psi_n):X_\mu\to\K^n$ with $j^\mu\psi_j=j^\mu\vp_j$, $j=1,\dots,n$, $\psi$ is flat at zero.
\item[(iii')] There exist $\mu\geq\mu_0$ and an analytic mapping $\psi=(\psi_1,\dots,\psi_n):X_\mu\to\K^n$ such that $\psi$ is flat at zero and $j^\mu\psi_j=j^\mu\vp_j$, $j=1,\dots,n$.
\end{itemize}
\end{theorem}

\begin{proof}
Let $J=(h_1(x),\dots,h_s(x),y_1-\vp_1(x),\dots,y_n-\vp_n(x))\cdot\K\{y,x\}$, and let $\mu_0=\dim_\K(\K\{x\}/J(0))$.
Let $\psi=(\psi_1,\dots,\psi_n):X\to\K^n$ be an arbitrary analytic mapping with $j^{\mu_0}\psi_j=j^{\mu_0}\vp_j$, $j=1,\dots,n$, and let
\[
Q\coloneqq(h_1(x),\dots,h_s(x),y_1-\psi_1(x),\dots,y_n-\psi_n(x))\cdot\K\{y,x\}\,.
\]
We shall show that then $\NN(Q(0))=\NN(J(0))$. This proves that $\dim_\K(\K\{x\}/Q(0))=\mu_0$ (hence $\psi$ is a finite mapping) and, consequently, the theorem follows directly from Proposition~\ref{prop:flat-fin-det-finite} and Remark~\ref{rem:Milnor-will-suffice}.

First, let us show that $\NN(Q(0))$ contains all the vertices of $\NN(J(0))$. Given $\beta'$ a vertex of $\NN(J(0))$, let $F\in J$ be such that $\inexp(F(0))=\beta'$. Write $F=\sum_{i=1}^sg_i\cdot h_i+\sum_{j=1}^nq_j\cdot(y_j-\vp_j)$, for some $g_i,q_j\in\K\{y,x\}$. Define $G=\sum_{i=1}^sg_i\cdot h_i+\sum_{j=1}^nq_j\cdot(y_j-\psi_j)$. Then $G\in Q$ and, by assumption, $j^\mu G=j^\mu F$.
In particular, $\inexp(G(0))=\inexp(F(0))=\beta'$, as $|\beta'|\leq\mu$. It follows that $\beta'\in\NN(Q(0))$.

Now, choose $\beta''\in\N^m\setminus\NN(J(0))$ and suppose that $\beta''\in\NN(Q(0))$. Pick $G\in Q$ such that $\inexp(G(0))=\beta''$, and write $G=\sum_{i=1}^sg_i\cdot h_i+\sum_{j=1}^nq_j\cdot(y_j-\psi_j)$, for some $g_i,q_j\in\K\{y,x\}$. Define $F=\sum_{i=1}^sg_i\cdot h_i+\sum_{j=1}^nq_j\cdot(y_j-\vp_j)$. Then $F\in J$, and $j^\mu F=j^\mu G$, by assumption. Since $|\beta''|\leq\mu$, we get $\inexp(F(0))=\inexp(G(0))=\beta''$, which contradicts the choice of $\beta''$.
\end{proof}

\begin{remark}
\label{rem:not-for-modules}
Note that Theorem~\ref{thm:flat-fin-det-finite} cannot be generalized to finitely generated modules.
More precisely, suppose that a finite $\K\{y\}$-module $M$ is given as the cokernel of a homomorphism $\Phi:\K\{y\}^q\to\K\{y\}^p$; i.e., $M=\K\{y\}^p/(\Phi(e_1),\dots,\Phi(e_q))$, where $\{e_1,\dots,e_q\}$ is the canonical basis of the free $\K\{y\}$-module $\K\{y\}^q$.
For $\mu\in\N$, let $M_\mu$ denote the module $\K\{y\}^p/(j^\mu\Phi(e_1),\dots,j^\mu\Phi(e_q))$.
One might expect that flatness of $M$ over $\K\{y\}$ is equivalent to flatness of $M_\mu$ for sufficiently large $\mu$. This is not the case however, as the following example shows.
\end{remark}

\begin{example}
\label{ex:not-for-modules}
Consider $f_1$ and $f_2$ from Example~\ref{ex:no-equality}, and let $M$ be a $\K\{x,y\}$-submodule of $\K\{x,y\}^3$ generated by
\[
G_1=(1,0,f_1), \ G_2=(1,1,f_2), \mathrm{\ and\ \,} G_3=(y^2-x,-x,0).
\]
Let $\OO$ denote the structure sheaf of $\K^2$ and let $\MM$ be a coherent $\OO$-module whose stalk at the origin is $\MM_{(0,0)}=\K\{x,y\}^3/M$. We have $G_1(0,0)=(1,0,0)$, $G_2(0,0)=(1,1,0)$, and $G_3(0,0)=(0,0,0)$, hence the multiplicity of $\MM$ at the origin is $3-2=1$. On the other hand, $G_3$ being a combination of $G_1$ and $G_2$ (indeed, $G_3=y^2\cdot G_1-x\cdot G_2$), the multiplicity of $\MM$ at any other point cannot be less than $1$. Therefore, multiplicity of $\MM$ is constant, and so $\MM$ is flat at $(0,0)$ (see, e.g., \cite[Thm.\,1.78]{GLS}).

Let now $\MM^\mu$ be a coherent $\OO$-module with $\MM^\mu_{(0,0)}=\K\{x,y\}^3/M_\mu$, where $M_\mu$ is generated by $j^\mu G_1$, $j^\mu G_2$, and $j^\mu G_3$. We claim that $\MM^\mu$ is not flat at the origin (equivalently, $\K\{x,y\}^3/M_\mu$ is not $\K\{x,y\}$-flat), for any $\mu\geq5$.

Indeed, identifying $M_\mu$ with a matrix with rows $j^\mu G_i$, $i=1,2,3$, we get
\[
\det(M_\mu)=y^2\cdot j^\mu f_1-x\cdot j^\mu f_2\,,
\]
which we know is equal to $xy^{\mu+1}$, provided $\mu\geq5$ (cf. Example~\ref{ex:no-equality}).
Therefore, for any $a\neq0$ and $b\neq0$, the stalk $\MM^\mu_{(a,b)}$ has multiplicity zero, while $\MM^\mu_{(0,0)}$ has multiplicity one.
\end{example}
\medskip

\subsection{Mappings from complete intersections}

\subsubsection*{Proof of Theorem~\ref{thm:flat-fin-det-comp-int}}
Let $h_1,\dots,h_s$ be a regular sequence in $\K\{x\}$ for which $\OO_{X,0}=\K\{x\}/(h_1,\dots,h_s)$.
Identifying $X$ with the graph of $\vp$, we can write $\OO_{X,0}=\K\{y,x\}/J$, where
\[
J=(h_1(x),\dots,h_s(x),y_1-\vp_1(x),\dots,y_n-\vp_n(x))\cdot\K\{y,x\}\,.
\]
For $\mu\in\N$, set
\[
J_\mu\coloneqq(h_1(x),\dots,h_s(x),y_1-j^\mu\vp_1(x),\dots,y_n-j^\mu\vp_n(x))\cdot\K\{y,x\}.
\]
By a well-known flatness criterion over regular local rings (see, e.g., \cite[Thm.\,B.8.11]{GLS}) and because the local ring $\OO_{X,0}$ is Cohen-Macaulay, flatness of $\vp$ at zero is equivalent to
\begin{equation}
\label{eq:C-M}
\dim(\K\{x\}/J(0))=\dim\OO_{X,0}-n=(m-s)-n.
\end{equation}
Suppose then that $\vp$ is flat at zero. 
By \cite[Ch.\,2, Prop.\,5.3]{Tou}, there exists $\mu'\geq1$ such that $\dim(\K\{x\}/J_\mu(0))\,\leq\,\dim(\K\{x\}/J(0))$, for all $\mu\geq\mu'$. On the other hand, as an ideal generated by $s+n$ elements, every $J_\mu(0)$ satisfies inequality $\dim(\K\{x\}/J_\mu(0))\geq m-(s+n)$.
Therefore, by \eqref{eq:C-M}, $\dim(\K\{x\}/J_\mu(0))=(m-s)-n$, and so $\vp_\mu$ is flat at zero. This proves (i)\,$\Rightarrow$\,(ii).

The implication (ii)\,$\Rightarrow$\,(iii) is trivial.
For the proof of (iii)\,$\Rightarrow$\,(i), suppose that $\vp$ is not flat at the origin. Then, by Proposition~\ref{prop:Hir2}, one can choose a nonzero $F\in J$ supported outside of $\NN(J(0))$. In particular, $\inexp(F)\notin\NN(J(0))$. Write
 $F=\sum_{i=1}^sg_ih_i+\sum_{k=1}^nq_k\cdot(y_k-\vp_k)$, for some $g_i,q_k\in\K\{y,x\}$, and set $\mu''\coloneqq\max\{|\inexp(F)|,l_0\}$, where $l_0$ is the maximum of lengths of vertices of the diagram $\NN(J(0))$. Fix $\mu\geq\mu''$, and define $G\coloneqq\sum_{i=1}^sg_ih_i+\sum_{k=1}^nq_k\cdot(y_k-j^\mu\vp_k)$. Then $G\in J_\mu$ and $\inexp(G)=\inexp(F)$ (as $j^\mu G=j^\mu F$ and by Remark~\ref{rem:jet-exp}). By Lemma~\ref{lem:diagram-up-to-l}, $\inexp(G)\notin\NN(J_\mu(0))$, which implies that $\vp_\mu$ is not flat (by Propositions~\ref{prop:Hir2} and~\ref{prop:Hir1}). The theorem thus holds for $\mu_0=\max\{\mu',\mu''\}$.
\qed

\begin{corollary}
\label{cor:1fin-det-comp-int}
Given a collection of power series $f_1,\dots,f_s\in\K\{x\}$, there exists $\mu_0\geq1$ such that the following conditions are equivalent:
\begin{itemize}
\item[(i)] $f_1,\dots,f_s$ form a regular sequence.
\item[(ii)] For all $\mu\geq\mu_0$, $j^\mu f_1,\dots,j^\mu f_s$ form a regular sequence.
\item[(iii)] There exists $\mu\geq\mu_0$ such that $j^\mu f_1,\dots,j^\mu f_s$ form a regular sequence.
\end{itemize}
\end{corollary}

\begin{proof}
The equivalence follows immediately from Theorem~\ref{thm:flat-fin-det-comp-int} applied to the mapping $\vp\coloneqq(f_1,\dots,f_s):\K^m\to\K^s$, and the fact that $h_1,\dots,h_k\in\K\{x\}$ form a regular sequence if and only if $(h_1,\dots,h_k):\K^m\to\K^k$ is flat (see, e.g., \cite[Thm.\,B.8.11]{GLS}).
\end{proof}

For the next result, we will need the following useful observation.

\begin{remark}
\label{rem:diagram-dim}
For an ideal $I$ in $\K\{x\}$, the following conditions are equivalent, up to a generic linear change of coordinates $x$:
\begin{itemize}
\item[(i)] $\dim(\K\{x\}/I)\leq\dim\K\{x\}-k$.
\item[(ii)] The diagram $\NN(I)$ has a vertex on each of the axes corresponding to $x_1,\dots,x_k$.
\end{itemize}
Indeed, condition (ii) clearly implies (i). On the other hand, (i) implies that (up to a generic linear change of coordinates) $\K\{x\}/I$ is a finite $\K\{\tilde{x}\}$-module, where $\tilde{x}=(x_{k+1},\dots,x_m)$. The latter is equivalent to saying that $I$ contains a distinguished pseudo-polynomial $f_j\in\K\{\tilde{x}\}[x_j]$, for every $j=1,\dots,k$ (see, e.g., \cite[Ch.\,III, Sec.\,2.2]{Loj}), hence (ii).
\end{remark}

\begin{corollary}
\label{cor:fin-det-comp-int}
Given a collection of power series $f_1,\dots,f_s\in\K\{x\}$, there exists $\mu_0\geq1$ such that the following conditions are equivalent:
\begin{itemize}
\item[(i)] $f_1,\dots,f_s$ form a regular sequence.
\item[(ii)] For every $\mu\geq\mu_0$, every sequence $g_1,\dots,g_s$ satisfying $j^\mu g_i=j^\mu f_i$ ($i=1,\dots,s$) is regular in $\K\{x\}$.
\item[(iii)] There exists $\mu\geq\mu_0$ such that every sequence $g_1,\dots,g_s$ satisfying $j^\mu g_i=j^\mu f_i$ ($i=1,\dots,s$) is regular in $\K\{x\}$.
\end{itemize}
\end{corollary}

\begin{proof}
Suppose that $f_1,\dots,f_s$ form a regular sequence in $\K\{x\}$. Then $\dim\K\{x\}/$ $(f_1,\dots,f_s)=m-s$.
By Remark~\ref{rem:diagram-dim}, one can assume that the diagram $\NN(I)$ of the ideal $I=(f_1,\dots,f_s)\cdot\K\{x\}$ has a vertex on each of the axes corresponding to $x_1,\dots,x_s$; say, $\beta^1,\dots,\beta^s$. Let $\mu'$ denote the maximum of lengths of $\beta^1,\dots,\beta^s$. Fix $\mu\geq\mu'$ and choose arbitrary $g_1,\dots,g_s\in\K\{x\}$ with $j^\mu g_i=j^\mu f_i$ for $i=1,\dots,s$. Set $J\coloneqq(g_1,\dots,g_s)\cdot\K\{x\}$. For $k=1,\dots,s$, let $F_k\in I$ be such that $\inexp(F_k)=\beta^k$. Write $F_k=\sum_{i=1}^sq_{ki}f_i$, for some $q_{ki}\in\K\{x\}$. Define $G_k\coloneqq\sum_{i=1}^sq_{ki}g_i$, $k=1,\dots,s$. Then $G_k\in J$ and $j^\mu G_k=j^\mu F_k$ for all $k$. As $|\beta^k|\leq\mu'\leq\mu$, it follows (by Remark~\ref{rem:jet-exp}) that $\inexp(G_k)=\inexp(F_k)=\beta^k$, and so $\beta^1,\dots,\beta^s$ all belong to $\NN(J)$. Hence, by Remark~\ref{rem:diagram-dim} again, $\dim\K\{x\}/J\leq m-s$. Thus, $g_1,\dots,g_s$ form a regular sequence, which proves (i)\,$\Rightarrow$\,(ii).

The implication (ii)\,$\Rightarrow$\,(iii) is trivial.
For the proof of (iii)\,$\Rightarrow$\,(i), suppose that $f_1,\dots,f_s$ do not form a regular sequence. Then, by Corollary~\ref{cor:1fin-det-comp-int}, there exists $\mu''$ such that, for all $\mu\geq\mu''$, $j^\mu f_1,\dots,j^\mu f_s$ do not form a regular sequence. The corollary thus holds for $\mu_0\coloneqq\max\{\mu',\mu''\}$.
\end{proof}

It is now easy to see that flatness of mappings from complete intersections is finitely determined, in the following sense.

\begin{theorem}
\label{thm:2flat-fin-det-comp-int}
Let $X$ be a $\K$-analytic subspace of $\K^m$. Suppose that $0\in X$ and the local ring $\OO_{X,0}$ is a complete intersection.
Let $\vp=(\vp_1,\dots,\vp_n):X\to\K^n$, $\vp(0)=0$, be a $\K$-analytic mapping. Then, there exists $\mu_0\geq1$ such that the following conditions are equivalent:
\begin{itemize}
\item[(i)] $\vp$ is flat at zero.
\item[(ii)] For every $\mu\geq\mu_0$, every analytic mapping $\psi=(\psi_1,\dots,\psi_n):X\to\K^n$ with $j^\mu\psi_i=j^\mu\vp_i$ ($i=1,\dots,n$) is flat at zero.
\item[(iii)] There exists $\mu\geq\mu_0$ such that every analytic mapping $\psi=(\psi_1,\dots,\psi_n):X\to\K^n$ with $j^\mu\psi_i=j^\mu\vp_i$ ($i=1,\dots,n$) is flat at zero.
\end{itemize}
If, moreover, $X$ is defined at $0\in\K^m$ by a regular sequence $h_1,\dots,h_s$ and, for $\mu\in\N$, $X_\mu$ denotes a local model defined at $0\in\K^m$ by $\OO_{X_\mu,0}=\K\{x\}/(j^\mu h_1,\dots,j^\mu h_s)$, then the above conditions are equivalent to the following:
\begin{itemize}
\item[(ii')] For every $\mu\geq\mu_0$, every analytic mapping $\psi=(\psi_1,\dots,\psi_n):X_\mu\to\K^n$ with $j^\mu\psi_i=j^\mu\vp_i$ ($i=1,\dots,n$) is flat at zero.
\item[(iii')] There exists $\mu\geq\mu_0$ such that every analytic mapping $\psi=(\psi_1,\dots,\psi_n):X_\mu\to\K^n$ with $j^\mu\psi_i=j^\mu\vp_i$ ($i=1,\dots,n$) is flat at zero.
\end{itemize}
\end{theorem}

\begin{proof}
Recall (\cite[Thm.\,B.8.11]{GLS}) that an analytic mapping $\vp=(\vp_1,\dots,\vp_n):X\to\K^n$ is flat at zero if and only if $\vp_1,\dots,\vp_n$ form an $\OO_{X,0}$-regular sequence. Therefore, if $X$ is defined at $0\in\K^m$ by a regular sequence $h_1,\dots,h_s$, then the latter is equivalent to saying that $h_1,\dots,h_s,\vp_1,\dots,\vp_n$ form a $\K\{x\}$-regular sequence.
The theorem thus follows from Corollary~\ref{cor:fin-det-comp-int} applied to the sequence $h_1,\dots,h_s,\vp_1,\dots,\vp_n$.
\end{proof}

\begin{question}
\label{q:not-complete}
It would be interesting to know if one could relax the complete intersection assumption on the domain $X$ in Theorem~\ref{thm:2flat-fin-det-comp-int}.
\end{question}

\section{Stability under deformations}
\label{sec:stability}

Let $\vp:X\to Y$ be a morphism of $\K$-analytic spaces, and let $(T,0)$ be a pointed space (i.e., a $\K$-analytic space with a distinguished point $0$). By a \emph{deformation of $\vp$ over $Y$} we shall mean a Cartesian diagram of the form
\[
\begin{CD}
X    @>>>    \XX \\
@V{\vp}VV    @VV{\Phi}V \\
Y     @>>>   Y\times T \\
@VVV         @VV{\pi}V \\
\{0\} @>>>   T,
\end{CD}
\]
such that $\XX$ is flat over $T$ (see, e.g., \cite{GLS}). If $\XX$ is of the form $X\times T$ then $\Phi$ is called an \emph{unfolding} of $\vp$.

In this section, we study stability of flatness of $\vp$ under such deformations. Since our considerations are local, we can assume from the start that all spaces are local models; say, $\XX\subset\K^m$, $Y\subset\K^n$, $T\subset\K^k$, $X$ and $Y$ pass through the origins in $\K^m$ and $\K^n$ respectively, and $\vp(0)=0$. Let $x=(x_1,\dots,x_m)$, $y=(y_1,\dots,y_n)$, and $t=(t_1,\dots,t_k)$ denote the systems of variables in the respective ambient spaces of $\XX$, $Y$ and $T$.

For $\theta\in T$, we shall denote by $\Pt$ the specialization of $\Phi$ over $\theta$, that is, the unique mapping closing the following Cartesian diagram
\[
\tag{\dag}
\begin{CD}
(\pi\circ\Phi)^{-1}(\theta)  @>>>   \XX \\
@V{\Phi^\theta}VV            @VV{\Phi}V \\
Y           @>>>             Y\times T \\
@VVV                         @VV{\pi}V \\
\{\theta\}  @>>>             T.
\end{CD}
\]
\medskip

For the proof of Theorem~\ref{thm:stab-flat}, we shall first settle the case of smooth one-dimensional parameter space.

\begin{lemma}
\label{lem:stab-flat}
Let $\vp:X\to Y$ be a flat morphism of $\K$-analytic spaces, with $Y$ locally irreducible. If $\Phi$ is a deformation of $\vp$ parametrized by $T=\K$, then $\Pt$ is a flat mapping for every $\theta\in T$ near zero.
\end{lemma}

\begin{proof}
Since flatness is an open property (see, e.g., \cite[Thm.\,7.15]{BM}) and because flatness is preserved by base change (see, e.g., \cite[Prop.\,6.8]{H}), it suffices to show that $\Phi$ is flat at $0\in\XX$.
For a proof by contradiction, suppose that $\Phi$ is not flat at the origin.
Let $R$ denote the local ring $\OO_{Y,0}$. After identifying $\XX$ with the graph of $\Phi$, we can regard $\OO_{\XX,0}$ as a quotient of $\OO_{Y\times T,(0,0)}$; that is, $\OO_{\XX,0}=R\{t,x\}/I$ for some ideal $I=I(y,t,x)$ in $R\{t,x\}$.

Now, by Proposition~\ref{prop:Hir2}, there exists a nonzero $F\in I$ such that $\supp(F)\subset\N^m\setminus\NN(I(0,0,x))$, where the evaluation is at $(y,t)=(0,0)$. We have $F(y,0,x)\in I(y,0,x)$. Hence, flatness of $\vp$ at zero implies that $F(y,0,x)=0$ in $R\{x\}$, by Proposition~\ref{prop:Hir2} again. The latter means that $t$ divides $F$ in $R\{t,x\}$.
Let $d$ be the maximum power of $t$ which can be factored out from $F$ in $R\{t,x\}$, and set $\tilde{F}\coloneqq t^{-d}\cdot F$.
Then $\tilde{F}(y,0,x)$ is not zero anymore. But $\supp(\tilde{F})=\supp(F)$, so applying Proposition~\ref{prop:Hir2} once more, we get that $\tilde{F}\notin I$. Consequently, $t^d$ is a zerodivisor in the local ring $\OO_{\XX,0}$, which contradicts the flatness of $\pi\circ\Phi$.
\end{proof}

\subsubsection*{Proof of Theorem~\ref{thm:stab-flat}}
As in the proof of the above lemma, it suffices to show that $\Phi$ is flat at $0\in\XX$.
Suppose first that $T$ is smooth. The problem being local, we can thus assume that $T=\K^k$. We will prove by induction on $k$ that flatness of $\pi\circ\Phi:\XX\to Y\times\K^k\to\K^k$ and $\Phi^0=\vp:X\to Y$ at zero implies that $\Phi:\XX\to Y\times\K^k$ is flat at zero.

For $k=0$ there is nothing to prove, so suppose that $k\geq1$ and the statement holds for $k-1$. 
Consider the flat mapping $\XX_1\to Y\times\K^{k-1}\to\K^{k-1}$, defined as the pullback of the flat $\pi\circ\Phi$ by the inclusion $\K^{k-1}\hookrightarrow\K^k$.
The inductive hypothesis implies that $\Phi_1:\XX_1\to Y\times\K^{k-1}$ is flat.
Next, consider the mapping $\XX\to(Y\times\K^{k-1})\times\K\to\K$, which is flat as the composite of $\pi\circ\Phi$ with the projection $\K^k\to\K$. Applying Lemma~\ref{lem:stab-flat} to the Cartesian diagram
\[
\begin{CD}
\XX_1               @>>>   \XX \\
@V{\Phi_1}VV                              @VV{\Phi}V \\
Y\times\K^{k-1}     @>>>   (Y\times\K^{k-1})\times\K \\
@VVV                                       @VVV \\
\{0\}               @>>>   \K,
\end{CD}
\]
we conclude that $\Phi$ is flat at zero.
\smallskip

Finally, consider a general locally irreducible $T$. In this case, one can find a nonsingular $\K$-analytic space $Z$ and a dominant mapping $\sigma:Z\to T$, $\sigma(0)=0$, with $\dim Z_0=\dim T_0$ (for example, take a desingularization of $T$ near the origin). Consider the pullback of $\pi\circ\Phi:\XX\to Y\times T\to T$ by $\sigma:Z\to T$ (which is flat by the flatness of $\pi\circ\Phi$). One can easily check that this mapping factors as $\XX'\stackrel{\Phi'}{\rightarrow}Y\times Z\to Z$. By the first part of the proof, we thus get that $\Phi':\XX'\to Y\times Z$ is flat at zero. Moreover, the pullback of $\sigma$ by $\pi$ is clearly dominant, and so we get the following Cartesian square in which $\Phi'$ is flat and the bottom arrow is dominant:
\[
\begin{CD}
\XX           @<<<   \XX' \\
@V{\Phi}VV           @VV{\Phi'}V \\
Y\times T     @<<<   Y\times Z.
\end{CD}
\]
By assumption, $Y\times T$ is irreducible at $(0,0)$. Hence, the analytic flatness descent (see \cite[Prop.\,2.1]{AS1}) implies that $\Phi$ is flat at zero, as required.
\qed

\section{Complex case}
\label{sec:cplx}

In this section, we consider the case $\K=\C$. In the complex setting, flatness of a mapping $\vp:X\to Y$ has a natural geometric interpretation. Namely, it is equivalent to continuity in the family of fibres of $\vp$. In fact, if $Y$ is nonsingular and (the local ring of) $X$ is Cohen-Macaulay (at every point), then flatness of $\vp$ is equivalent to openness (see, e.g., \cite[\S\,3.20]{F}), and the latter simply means that all fibres of $\vp$ are of the same dimension.
In particular, over $\K=\C$, in Theorems~\ref{thm:flat-fin-det-comp-int} and~\ref{thm:2flat-fin-det-comp-int} `flatness' can be replaced with `openness', since complete intersections are Cohen-Macaulay.

Over singular targets the picture is (considerably) more complicated, nonetheless it is still possible to interpret flatness in purely geometric terms. As we show in \cite{AS2}, a morphism $\vp:X\to Y$ of complex-analytic spaces (with $Y$ locally irreducible) is flat at a point $\xi\in X$ if and only if every irreducible component of the fibred product $X\times_YZ$ at $(\xi,\zeta)$ is dominant over $Z_\zeta$, where $\sigma_\zeta:Z_\zeta\to Y_\eta$ is the local blowing up of $Y$ at $\eta=\vp(\xi)$.

Below, we generalize this idea and construct test mappings to detect higher order discontinuities in the family of fibres of a given mapping. For an analytic mapping $\vp:X\to Y$ with locally irreducible $Y$ and $X$ of pure dimension, one can speak of the generic fibre dimension of $\vp$, denoted $\lambda_\vp$. Further, let $\kappa_\vp$ be the maximum fibre dimension of $\vp$. We shall call the difference $\kappa_\vp-\lambda_\vp$ the \emph{fibre defect} of $\vp$. By the Remmert Open Mapping Theorem (see, e.g., \cite[Ch.\,V, \S\,6, Thm.\,2]{Loj}), $\vp$ is open if and only if its fibre defect is zero.

\subsection{Test mappings}

Consider a morphism $\vp:X\to Y$ of local models. Suppose that $X\subset\C^m$ is of pure dimension, $Y\subset\C^n$ is locally irreducible (of positive dimension), and $\vp(0)=0$. Suppose further that $n=\mathrm{edim}\OO_{Y,0}$. After a linear change of coordinates in $\C^n$ if needed, we can assume that $y_n$ belongs to the tangent cone of $Y$ at $0$. The following proposition gives a method of testing for the degree of fibre defect of a given mapping.

For $k\in\{0,\dots,n-1\}$, let $\sigma_k:\C^n\to\C^n$ denote the mapping
\[
(y_1,\dots,y_n)\mapsto(y_1,\dots,y_k,y_{k+1}y_n,\dots,y_{n-1}y_n,y_n)\,.
\]
In other words, $\sigma_k$ is the restriction to the affine chart $\{y_n\neq0\}$ of the blowing up of $\C^n$ with centre $C_k=\{y_{k+1}=\dots=y_n=0\}$. Denote by $Y_k^{st}$ the strict transform of $Y$ under $\sigma_k$.
We shall consider a Cartesian square of the form
\[
\begin{CD}
X\times_Y Y_k^{st}              @>>>    X \\
@V{\vp'_k}VV                    @VV{\vp}V \\
Y_k^{st}  @>{{\sigma_k}|_{Y_k^{st}}}>>  Y.
\end{CD}
\]

\begin{proposition}
\label{prop:fibre-defect-test}
Suppose that $\vp$ has fibre defect greater than $\delta$. Then:
\begin{itemize}
\item[(i)] The fibre $({\sigma_\delta}|_{Y_\delta^{st}}\circ\vp'_\delta)^{-1}(0)$ has at $(0,0)\in X\times_Y Y_\delta^{st}$ dimension greater than or equal to $\dim X$.
\item[(ii)] The fibred product $X\times_Y Y_\delta^{st}$ has an isolated irreducible component at $(0,0)$ which is mapped by ${\sigma_\delta}|_{Y_\delta^{st}}\circ\vp'_\delta$ into $C_\delta\cap Y$. Equivalently, $y_n$ is a zerodivisor in the reduced local ring $(\OO_{X\times_Y Y_\delta^{st},(0,0)})_{\mathrm{red}}$.
\end{itemize}
\end{proposition}

\begin{proof}
Let $\delta\in\N$ and suppose that $\kappa_\vp-\lambda_\vp>\delta$.
Since $\mathrm{edim}\OO_{Y,0}=n$, it follows that $Y_0$ is not contained in the germ of the center $(C_\delta)_0$.
Consequently,
\[
\dim(Y_\delta^{st}\cap\sigma_\delta^{-1}(C_\delta))=\dim({\sigma_\delta}|_{Y_\delta^{st}})^{-1}(C_\delta)=\dim{Y_\delta^{st}}-1=\dim Y-1\,.
\]
As the center itself is of dimension $\delta$, we get that
\[
\dim({\sigma_\delta}|_{Y_\delta^{st}})^{-1}(0)\geq(\dim Y-1)-\delta\geq\dim X-\lambda_\vp-1-\delta\,,
\]
and hence
\[
\fbd_{(0,0)}({\sigma_\delta}|_{Y_\delta^{st}}\circ\vp'_\delta)\geq(\dim X-\lambda_\vp-\delta-1)+\kappa_\vp\geq\dim X\,,
\]
which proves property (i).

On the other hand, since $\sigma_\delta$ is a biholomorphism outside $\sigma_\delta^{-1}(C_\delta)$, it follows that $\dim_{(\xi,\eta)}X\times_Y Y_\delta^{st}=\dim X$ for all $(\xi,\eta)$ except at most those for which $\eta$ is mapped by $\sigma_\delta$ into $C_\delta$. Therefore, either the fibre $({\sigma_\delta}|_{Y_\delta^{st}}\circ\vp'_\delta)^{-1}(0)$ itself contains an irreducible component of $X\times_Y Y_\delta^{st}$ at $(0,0)$ or else it is contained in a component mapped into $C_\delta\cap Y$.

The last statement of the proposition follows from the fact that the $\OO_{Y,0}$-module structure of $\OO_{X\times_Y Y_\delta^{st},(0,0)}$ factors as $\OO_{Y,0}\to\OO_{Y_\delta^{st},0}\to\OO_{X\times_Y Y_\delta^{st},(0,0)}$ and the image in $\OO_{Y_\delta^{st},0}$ of the ideal defining $C_\delta$ is the principal ideal generated by $y_n$.
\end{proof}

\begin{remark}
\label{rem:when-equivalent}
It is evident from the proof above that if $Y=\C^n$ and the mapping $\vp:X\to Y$ is dominant (i.e., $\lambda_\vp=\dim X-n$) then, conversely, the equivalent conditions (i) and (ii) of the proposition imply that the fibre defect of $\vp$ is greater than $\delta$.
\end{remark}
\medskip

\subsection{Stability of openness}

In this section we prove that, like flatness, openness of complex-analytic mappings is stable under deformations.
This follows from Theorem~\ref{thm:stab-flat} for mappings from Cohen-Macaulay into smooth spaces (by \cite[\S\,3.20]{F}), but in general an open mapping need not be flat.
We have the following proposition.

\begin{proposition}
\label{prop:stab-defect}
Let $\vp:X\to Y$ be a morphism of local models, where $Y\subset\C^n$ is locally irreducible, $X\subset\C^m$ is of pure dimension, and $\vp(0)=0$.
Suppose that $T$ is locally irreducible. Let $d\geq1$, and suppose that $\vp$ is a dominant mapping with fibre defect less than $d$. If $\Phi$ is a deformation of $\vp$ over $Y$, parametrized by $T$ and with $\XX$ of pure dimension, then $\Pt$ is dominant and has fibre defect less than $d$ for every $\theta\in T$ near zero. (In particular, this is the case if $\Phi$ is an unfolding of $\vp$.)
\end{proposition}

\begin{proof}
Set $l\coloneqq\dim T$ and $r\coloneqq\dim\XX$. As a flat mapping, $\pi\circ\Phi$ is open (by Douady \cite{Dou}), and hence its fibre dimension is $r-l$ at every point $\xi\in\XX$. In particular, $X=(\pi\circ\Phi)^{-1}(0)$ is of pure dimension $r-l$. By dominance of $\vp$, we have $\dim Y=\dim X-\lambda_\vp$, that is, $\dim Y=r-l-\lambda_\vp$. Since $\lambda_\vp$ is the generic fibre dimension of $\vp$, it follows that $Y$ contains an open subset $Z$ adherent to $0\in\C^n$ such that for all $\eta\in Z$, $\dim\vp^{-1}(\eta)=\lambda_\vp$. As $\vp^{-1}(\eta)=\Phi^{-1}(\eta,0)$, it follows by upper semicontinuity of fibre dimension of $\Phi$ that $\lambda_{\Pt}\leq\lambda_\vp$ for $\theta\in T$ near zero.

On the other hand, one always has $\lambda_{\Pt}\geq\dim(\pi\circ\Phi)^{-1}(\theta)-\dim Y=r-l-\dim Y$. Therefore, $\lambda_{\Pt}\geq(r-l)-(r-l-\lambda_\vp)=\lambda_\vp$, and so $\lambda_{\Pt}=\lambda_\vp$ for all $\theta\in T$ near zero.
Since all $(\pi\circ\Phi)^{-1}(\theta)$ are of the same dimension as $X$, the dominance of $\Pt$ follows.

Finally, $\kappa_\vp$ is equal to the dimension of $\vp^{-1}(0)$, and hence $\kappa_{\Pt}\leq\kappa_\vp$, by upper semicontinuity of fibre dimension of $\Phi$ again. Thus, for all $\theta\in T$ near zero, $\kappa_{\Pt}-\lambda_{\Pt}\leq\kappa_\vp-\lambda_\vp<d$, as required.
\end{proof}

\begin{corollary}
\label{cor:stab-open}
Openness is stable under deformations:
If $\vp$ is an open mapping and $\Phi$ is its deformation as in Proposition~\ref{prop:stab-defect}, then $\Phi^\theta$ is open for every $\theta\in T$ near zero.
\end{corollary}

\subsection{Finite determinacy of flatness of complex-analytic mappings}

We conclude the paper with a comment on Theorem~\ref{thm:2flat-fin-det-comp-int} in the complex case.

Recall that, for a $d$-dimensional complex analytic set $X$ in $\C^m$ and a point $\xi\in X$, one defines the \emph{multiplicity} $\mu_\xi(X)$ of $X$ at $\xi$ as follows: In a generic system of coordinates $x$ at $\xi$ in $\C^m$, the local ring $\OXxi$ is a finite $\C\{\tilde{x}\}$-module, where $\tilde{x}=(x_1,\dots,x_d)$. We set $\mu_\xi(X)$ to be the rank of this module. Equivalently, $\mu_\xi(X)$ is the generic cardinality of a fibre of a projection of $X$ to a generic $d$-dimensional linear subspace of $\C^m$ (in a neighbourhood of $\xi$).

\begin{proposition}
\label{prop:mult-will-do}
Under the notations of Theorem~\ref{thm:2flat-fin-det-comp-int}, let $Z$ denote the fibre $\vp^{-1}(0)$ (that is, $Z_0=\VV(J)$, where $J=(h_1,\dots,h_s,\vp_1,\dots,\vp_n)$). Then, the implications $(i)\Rightarrow(ii)$ and $(i)\Rightarrow(iii)$ in Theorem~\ref{thm:2flat-fin-det-comp-int} hold with $\mu_0\coloneqq\mu_0(Z)$.
\end{proposition}

\begin{proof}
Indeed, directly from the definition of $\mu_0(Z)$ it follows that (after a linear change of variables $x$, if needed) the classes in $\C\{x\}/J$ of $x_1,\dots,x_{n+s}$ are integral over $\C\{\tilde{x}\}$, where $\tilde{x}=(x_{n+s+1},\dots,x_m)$. Hence, there are monic polynomials of minimal degree $P_1,\dots,P_{n+s}\in\C\{\tilde{x}\}[T]$, such that $P_i(x_i)\in J$, $i=1,\dots,n+s$.
Clearly, $\deg(P_i)\leq\mu_0(Z)$ for all $i$. Setting $L(\beta)=\sum_i\lambda_i\beta_i$ with $\lambda_1=\dots=\lambda_{n+s}=1$ and $\lambda_{n+s+1}=\dots=\lambda_m=\mu_0(Z)+1$, we get that $\NN_L(J)$ contains vertices (of lengths at most $\mu_0(Z)$) on the axes corresponding to $x_1,\dots,x_{n+s}$ (namely those of the leading terms of $P_1,\dots,P_{n+s}$).
The claim thus follows from the proof of Corollary~\ref{cor:fin-det-comp-int}.
\end{proof}

\bibliographystyle{amsplain}

\begin{thebibliography}{99}

\bibitem {AS1} J.\,Adamus and H.\,Seyedinejad,
 \textit{Flatness testing over singular bases}, Ann. Polon. Math. \textbf{107} (2013), 87--96.

\bibitem {AS2} J.\,Adamus and H.\,Seyedinejad,
 \textit{A fast flatness testing criterion in characteristic zero}, Proc. Amer. Math. Soc. \textbf{143} (2015), 2559--2570.

\bibitem {BM} E.\,Bierstone and P.\,D.\,Milman,
 The local geometry of analytic mappings, Dottorato di Ricerca in Matematica, ETS Editrice, Pisa, 1988.

\bibitem {Dou} A.\,Douady,
 \textit{Le probl\`eme des modules pour les sous-espaces analytiques compacts d'un espace analytique donn\'e}, Ann. Inst. Fourier
  (Grenoble) \textbf{16:1} (1966), 1--95.

\bibitem {F} G.\,Fischer,
 Complex analytic geometry, Lecture Notes in Math. \textbf{538}, Springer, Berlin, Heidelberg, New York, 1976.
 
\bibitem {GLS} G.-M.\,Greuel, C.\,Lossen and E.\,Shustin,
 Introduction to singularities and deformations, Springer Monographs in Mathematics, Springer, Berlin, 2007.
 
\bibitem {H} H.\,Hironaka,
 \textit{Stratification and flatness}, Real and Complex Singularities, Proc. Oslo 1976, ed. Per Holm, Sijthoff and Noordhoff (1977),
  199--265.

\bibitem {Loj} S.\,{\L}ojasiewicz,
 Introduction to Complex Analytic Geometry, Birkh\"{a}user, Basel, Boston, Berlin, 1991.

\bibitem {Tou} J.-Cl.\,Tougeron,
 Id{\'e}aux de fonctions diff{\'e}rentiables, Ergebnisse der Mathematik und ihrer Grenzgebiete, 71, Springer, Berlin-New York, 1972.
\end{thebibliography}

\end{document}